\newcommand{\dataversione}{\today}

\documentclass[11pt,reqno]{amsart}
\usepackage{hyperref}\hypersetup{colorlinks=true, citecolor=blue}
\usepackage{graphicx}
\usepackage{xfrac}
\usepackage{esint,amsfonts,amsmath,amssymb,epsfig,
mathrsfs,bm,accents,yfonts,mathtools,dsfont,stmaryrd}

\numberwithin{equation}{section}
\newtheoremstyle{mytheorem}
{}
{}
{\it}
{\parindent}
{\bf}
{.}
{ }
{\thmnumber{#2.~}\thmname{#1}\thmnote{~\rm#3}}

\newtheoremstyle{myremark}
{}
{}
{\rm}
{\parindent}
{\bf}
{.}
{ }
{\thmnumber{#2.~}\thmname{#1}\thmnote{~\rm#3}}

\newtheoremstyle{myparagraph}
{}
{}
{\rm}
{\parindent}
{\bf}
{}
{ }
{\thmnumber{#2.~}\thmname{#1}\thmnote{#3}}

\theoremstyle{mytheorem}

\newtheorem{theorem}[subsection]{Theorem}

\newtheorem{corollary}[subsection]{Corollary}

\newtheorem{proposition}[subsection]{Proposition}

\newtheorem{definition}[subsection]{Definition}
\newtheorem{example}[subsection]{Example}

\theoremstyle{myremark}
\newtheorem{remark}[subsection]{Remark}
\newtheorem*{remark*}{Remark}

\theoremstyle{myparagraph}

\newtheorem*{parag*}{}

\makeatletter
\def\@secnumfont{\sc}
\def\section{\@startsection{section}{1}%
\z@{1.5\linespacing\@plus .2\linespacing}{.7\linespacing}%
{\normalfont\sc\centering}}
\makeatother

\makeatletter
\def\ps@headings{\ps@empty
 \def\@evenhead{%
  \setTrue{runhead}%
  \normalfont\footnotesize
  \rlap{\thepage}\hfil
  \def\thanks{\protect\thanks@warning}%
  \leftmark{}{}\hfil}%
 \def\@oddhead{%
  \setTrue{runhead}%
  \normalfont\footnotesize\hfil
  \def\thanks{\protect\thanks@warning}%
  \rightmark{}{}\hfil \llap{\thepage}}%
\let\@mkboth\markboth}
\makeatother
\makeatletter
\renewenvironment{proof}[1][\proofname]{\par
  \pushQED{\qed}%
  \normalfont \topsep6\p@\@plus6\p@\relax
  \trivlist
  \itemindent\normalparindent
  \item[\hskip\labelsep
    \scshape
    #1\@addpunct{.}]\ignorespaces
}{%
  \popQED\endtrivlist\@endpefalse
}
\providecommand{\proofname}{Proof}
\makeatother
\newcommand{\R}{\mathbb{R}}

\newcommand\res{\mathop{\hbox{\vrule height 7pt width .5pt depth 0pt
\vrule height .5pt width 6pt depth 0pt}}\nolimits}

\newcommand{\cH}{{\mathcal{H}}}

\newcommand\Z{{\mathbb Z}}

\newcommand\N{{\mathbb N}}
\newcommand\Q{{\mathbb Q}}

\newcommand{\Haus}[1]{{\mathcal{H}}^{#1}} 

\begin{document}

	%
\pagestyle{empty}
\pagestyle{myheadings}
\markboth%
{\underline{\centerline{\hfill\footnotesize%
\textsc{A. Marchese \& A. Massaccesi}%
\vphantom{,}\hfill}}}%
{\underline{\centerline{\hfill\footnotesize%
\textsc{An optimal irrigation network}%
\vphantom{,}\hfill}}}

	%
\thispagestyle{empty}

~\vskip -1.1 cm

	%
\centerline{\footnotesize version: \dataversione%
\hfill
}

\vspace{1.7 cm}

	%
{\Large\sl\centering
An optimal irrigation network\\
with infinitely many branching points
\\
}

\vspace{.4 cm}

	%
\centerline{\sc Andrea Marchese \& Annalisa Massaccesi}

\vspace{.8 cm}

{\rightskip 1 cm
\leftskip 1 cm
\parindent 0 pt
\footnotesize

	%
{\sc Abstract.}
The Gilbert-Steiner problem is a mass transportation problem, where the cost of the transportation depends on the network used to move the mass and it is proportional to a certain power of the ``flow''. In this paper, we introduce a new formulation of the problem, which turns it into the minimization of a convex functional in a class of currents with coefficients in a group. This framework allows us to define calibrations, which can be used to prove the optimality of concrete configurations. We apply this technique to prove the optimality of a certain irrigation network, having the topological property mentioned in the title.
\par
}

\section{Introduction}
%
%

The Gilbert-Steiner problem is a variant of the optimal transport problem with a cost depending not only on the position of the masses at the beginning and at the end of the process, but also on the path that has been chosen to move them. As in the optimal transport problem, the Gilbert-Steiner's datum consists of a positive measure $\mu_-$ (source) and a positive measure $\mu_+$ (well), with the same total mass $\|\mu_-\|(\R^d)=\|\mu_+\|(\R^d)$. One wants to move $\mu_-$ onto $\mu_+$, minimizing the global cost. The cost of the transport is concave with respect to the density of the transported mass, as it is in many natural phenomena and engineering problems\footnote{See \cite{West_Brown_En,Mau_Fi_Wei_Sa,Bran_But,Xia4}, for instance, for the models arising from nature and \cite{Gil,Deb,Bha_Salz} for the models arising from engineering and optimal design. Actually, the Gilbert-Steiner problem is motivated not only by concrete applications. Indeed, it originated from operational research and graph theory, as one can see in \cite{Zang}, and other very interesting mathematical problems can be related with it, such as the study of topological singularities  for Ginzburg-Landau models (see \cite{Be}).}. 

In particular, if $\mu_-$ and $\mu_+$ are sums of Dirac masses on a set of sources $X$ and a set of wells $Y$, respectively, then a transporting network is described by a $1$-dimensional rectifiable set $\Sigma$ and a multiplicity function $\theta\in L^1(\Sigma;\Z)$. The set $\Sigma$ is a union of paths joining points of $X$ to points of $Y$ and the multiplicity represents the ``flow'' of the transported mass through each point of the set $\Sigma$. The cost associated to each transporting network is the integral on the set $\Sigma$ of $\theta^\alpha$ for some $\alpha\in(0,1)$. It is interesting to observe the limit cases: if $\alpha=0$, the problem corresponds to the minimization of the $1$-dimensional measure of the set $\Sigma$, if $\alpha=1$, we recover the Monge-Kantorovich energy and the Monge problem of optimal transport. When $\mu_-$ is supported on a single point, we talk about irrigation problem. One immediately understands that, by its own nature, the problem tends to favor large flows and therefore it easily leads to branched structures. 

Several descriptions of this problem have been given by many authors: we refer to \cite{Ber_Ca_Mo1} for a detailed overview of the subject, though different models have been proposed in \cite{Xia1,Mad_Mo_So,Ber_Ca_Mo2,Ste,Pa_Ste3}. Those models are summarized and proved to be equivalent in \cite{Mad_So,Mad_So2}. Another model is available in \cite{Bran_But_San}, with further developments and a comparison with the model proposed in \cite{Ber_Ca_Mo2} in \cite{Bra_San}.

Our starting point is Xia's approach in \cite{Xia1}: the problem is described in terms of minimization of a certain energy, depending on $\alpha$, in a family of transport paths, that is, normal $1$-dimensional currents with boundary $\mu_+-\mu_-$. An important feature which is evident in
this approach is the non-convex nature of the problem: the energy of a convex combination of two competitors may be strictly larger than the energy of each single competitor. 

One of the achievements of the present paper is a new
formulation of the Gilbert-Steiner problem which turns it into the minimization of a convex functional. By associating to the measures $\mu_-$ and $\mu_+$ a normed group $G$ and a 0-dimensional current $B$ with coefficients in $G$, we
manage to prove that the Gilbert-Steiner problem is equivalent to the problem of finding a mass minimizer, among all 1-dimensional currents $Z$ with coefficients in $G$ having boundary $\partial Z = B$. The group $G$ is chosen depending only on the quantity $\|\mu_-\|(\R^d)=\|\mu_+\|(\R^d)$. In this new framework we can
introduce the notion of calibration, that is, a functional analytic tool to prove that a given candidate is a minimizer.

Such an approach is not completely new: in \cite{AnMa}, we applied the same procedure to the study of the
Steiner tree problem, which corresponds to the limit case $\alpha=0$ of the Gilbert-Steiner problem, with an additional connectedness constraint. 

Let us explain with an easy example the rough idea of the new formulation. Assume we want to irrigate the two wells $y_1=(2,-1)$ and $y_2=(2,1)$ from a (unique) source at the origin, which generates a flow of intensity 2. Fix $\alpha=1/2$. The classical description of a competitor for this problem consists of a superposition of two paths $\Gamma_1$ and $\Gamma_2$, both with multiplicity 1: the first path goes from the origin to $y_1$, the second from the origin to $y_2$. 
\begin{figure}[htbp]
\begin{center}
\scalebox{1}{
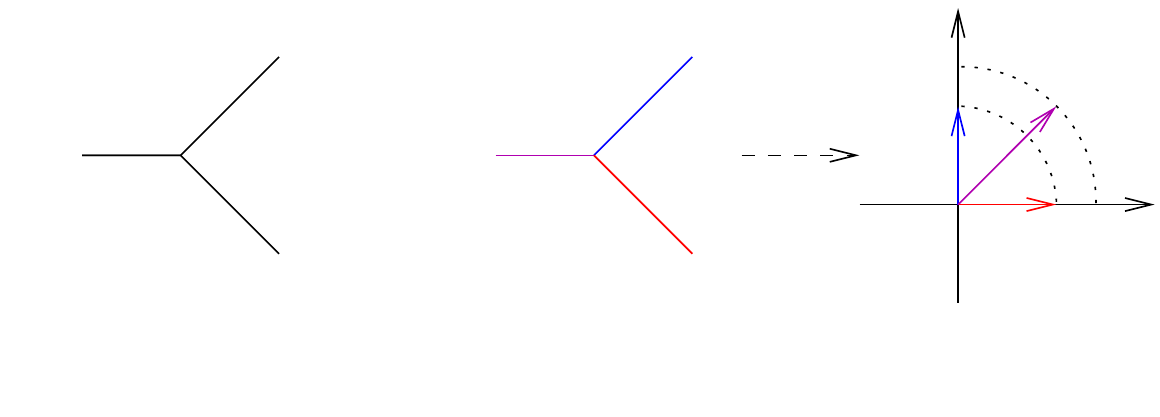
}
\end{center}
\end{figure}
Therefore the irrigation network can be identified with the set $\Gamma_1\cup\Gamma_2$ with a multiplicity, which is equal to 2 on $\Gamma_1\cap\Gamma_2$ and 1 elsewhere on $\Gamma_1\cup\Gamma_2$. One computes the energy of such competitor as 
$$\Haus{1}((\Gamma_1\cup\Gamma_2)\setminus(\Gamma_1\cap\Gamma_2))+\sqrt{2}\Haus{1}(\Gamma_1\cap\Gamma_2),$$
where $\Haus{1}$ denotes the 1-dimensional Hausdorff measure. 
In our formulation we associate to $\Gamma_1$ and $\Gamma_2$ vectors in $\R^2$ as multiplicities: $g_1$ and $g_2$, respectively, where $(g_1,g_2)$ is an orthonormal basis of $\R^2$. In this way the stretch ${\Gamma_1\cap\Gamma_2}$ is associated to the vector $g_1+g_2$. In order to compute the mass of the corresponding 1-current with coefficients in $\R^2$ one has to integrate on ${\Gamma_1\cup\Gamma_2}$ the Euclidean norm of the corresponding multiplicity. Clearly this mass coincides with the energy computed as above, because the Euclidean norm of $g_1+g_2$ is exactly $\sqrt{2}$.

In the last part of the paper we apply the equivalence result and the calibration method to prove the optimality of an irrigation tree with the topological property mentioned in the title of the paper, that is, an infinite number of branching points. The occurrence of such a topological behaviour for energy-minimizers has been the aim of various attempts: in \cite{Ber_Ca_Mo3} some characterizing principles for infinite irrigation trees are discussed, while in \cite{Pa_Ste_Tep} the authors prove the optimality of an analogous tree for the generalized Steiner tree problem (see also \cite{Pa_Ste2} and \cite{Pa_Uli}). Their proof relies on a
fine and rather delicate study of the geometry of the candidate minimizer. On the contrary, our proof is based on a calibration argument and we believe its simplicity to be an encouraging sign for further developments.


Let us describe now, in more detail, the content of each section.

In \S\ref{sec:not}, we collect those definitions and basic results in Geometric Measure Theory that are necessary to formulate both the discrete version of the Gilbert-Steiner problem, as introduced in \cite{Xia1} and our new
formulation in terms of 1-dimensional currents with coefficients in a group. Moreover we define a suitable notion of calibration for these objects and we prove that a calibrated current minimizes the mass in its homology class.

In \S\ref{sec:GS}, 
we recall and rework the discrete version of the Gilbert-Steiner problem given by \cite{Xia1}, in terms of an energy minimization problem in the family of classical integral currents, having a fixed boundary, which is supported on the given set of sources and wells (this is what we call problem (P1)). 
We also recall a structure theorem for integral 1-currents (see Theorem \ref{thm:struct_1currents}) which allows one to better understand why such an abstract class of competitors is indeed a natural choice.

In \S\ref{sec:equiv}, we describe how to associate to the given measures $\mu_-$ and $\mu_+$ a normed group $G$ and a mass minimization problem in a class of rectifiable currents with coefficients
in $G$ (which we call problem (P2)). In Theorem \ref{equiva} we prove the equivalence between problems (P1) and (P2). Then, we apply this result and the calibration technique introduced in \S\ref{sec:not} to
prove the optimality of some simple transportation networks. We conclude explaining how one should modify the two problems (P1) and (P2) in order to keep the equivalence, when the measures $\mu_-,\mu_+$, given as initial datum, are rescaled.

In \S\ref{sec:tree}, we describe an irrigation network on the separable Hilbert space $\ell^2$. Then, we prove that such a network is a solution of the continuous version of the Gilbert-Steiner problem associated to the corresponding boundary. The proof
is based on an argument of reduction to a finite dimensional framework, where we are able to exhibit a calibration and therefore to prove the optimality.

%
%
\section*{Acknowledgements}
We acknowledge Michael Goldman for inspiring discussions and helpful comments.

\section{Notation and Preliminaries}\label{sec:not}

In this first section, we briefly recall the main definitions and results concerning $1$-dimensional rectifiable currents with coefficients in the abelian normed group $\left(\Z^n,\|\cdot\|_\alpha\right)$. In
this
paper we limit ourselves to the essential facts which are unavoidable for the statement and the development of the main result. See \cite{AnMa} for a more general theory and further details.

\subsection{Differential forms with values in $\R^n$}
The ambient space where we want to treat the irrigation problem is $\R^d$. Concerning the coefficients, let us fix a standard basis $(g_1,\ldots,g_n)$ for $\R^n$. We fix $\alpha\in(0,1)$ and consider
the $n$-dimensional normed vector space
$(\R^n,\|\cdot\|_\alpha)$, with \[\|h\|_\alpha:=\left(\sum_{j=1}^n |h_j|^{\frac{1}{\alpha}}\right)^\alpha\] for every $h=h_1g_1+\ldots h_ng_n\in\R^n$. Once for all, we fix the (canonical) basis
 $(e_1,\ldots,e_n)$ of the dual space of $(\R^n,\|\cdot\|_\alpha)$ so that $\langle e_i,g_j\rangle=\delta_{ij}$.
With this notation, the dual space of
$(\R^n,\|\cdot\|_\alpha)$ is the $n$-dimensional normed vector space $(\R^n,\|\cdot\|_{1-\alpha})$.
%
\begin{definition}{\rm
An $\R^n$-\emph{valued covector} in $\R^d$ is a map $\omega:\Lambda_1(\R^d)\times\R^n\to\R$ with the following properties\footnote{In the following, we write $\langle\omega;\tau,h\rangle$ for $\omega(\tau,h)$.}:
\begin{itemize}
\item[(i)] $\langle\omega;\tau,\cdot\rangle\in \R^n$ for every vector $\tau\in\Lambda_1(\R^d)$;
\item[(ii)] $\langle\omega;\cdot,h\rangle\in\Lambda^1(\R^d)$ is a standard $1$-dimensional covector in $\R^d$ for every element $h\in\R^n$.
\end{itemize}
The linear space of $\R^n$-valued covectors in $\R^d$, endowed with the comass norm $$\|\omega\|^*_{\alpha}:=\sup_{|\tau|\le 1}\|\langle\omega;\tau,\cdot\rangle\|_{1-\alpha},$$ will be denoted by
$\Lambda^1_{(n,\alpha)}(\R^d)$. An $\R^n$-valued ($1$-dimensional) differential form defined on $\R^d$ is a map $\omega:\R^d\to\Lambda^1_{(n,\alpha)}(\R^d)$, its regularity
is the one inherited from the components
\[
\omega_j:=\langle\omega;\cdot,g_j\rangle:\R^d\to\Lambda^1(\R^d)\,,
\]
for $j=1,\ldots,n$.
}\end{definition}
\begin{definition}\label{def:de}{\rm
Consider a function $\psi:\R^d\to\R^n$ of class $\mathscr C^1$. We can compute the differential ${\rm d}\psi_j$ of each component $\psi_j$ of
$\psi$ ($j=1\,\ldots,n$), thus we will denote \[{\rm d}\psi={\rm d}\psi_1e_1+\ldots+{\rm d}\psi_ne_n\in\mathscr C(\R^d;\Lambda^1_{(n,\alpha)}(\R^d))\,.\]
}\end{definition}

\subsection{Rectifiable currents with coefficients in $\R^n$}

\begin{definition}{\rm
Given an integer number $n\ge 1$ and a real parameter $\alpha\in(0,1)$, we say that a map $Z:\mathscr C^\infty_c(\R^d;\Lambda^1_{(n,\alpha)}(\R^d))\to\R$ is a \emph{rectifiable current
with coefficients in the normed abelian group} $(\Z^n,\|\cdot\|_\alpha)$ if there exist a countably $1$-rectifiable set $\Sigma\subset \R^d$ (see \cite{Si}, for instance, for details about rectifiable sets), a vector field $\tau:\Sigma\to\Lambda_1(\R^d)$ which is almost
everywhere tangent to $\Sigma$ and an $L^1$ function $\theta:\Sigma\to(\Z^n,\|\cdot\|_\alpha)$ such that
\begin{itemize}
\item[(i)] for every $\omega\in\mathscr C^\infty_c(\R^d;\Lambda^1_{(n,\alpha)}(\R^d))$ we can write
\begin{equation}\label{def_curr}
Z(\omega)=\int_\Sigma \langle\omega(x);\tau(x),\theta(x)\rangle\,d\Haus{1}(x)\,;
\end{equation}
\item[(ii)] the \emph{mass} of the current is finite, that is
\[
\mathds M^{(n,\alpha)}(Z):=\int_\Sigma \|\theta(x)\|_\alpha\,d\Haus{1}(x)<+\infty\,.
\]
\end{itemize}
Moreover, we define the \emph{boundary} of $Z$ as the linear functional $\partial Z:\mathscr C^\infty_c(\R^d;\R^n)\to \R$ satisfying $\partial Z(\psi):=Z({\rm d}\psi)$ for every $\psi\in C^\infty_c(\R^d;\R^n)$. We require that
there exist a finite set of points $X=\{x_1,\ldots,x_m\}\subset\R^d$ and a function $\eta:X\to (\Z^n,\|\cdot\|_\alpha)$ such that
\begin{itemize}
\item[(iii)] for every $\psi\in C^\infty_c(\R^d;\R^n)$ we can write
\begin{equation}\label{def_bord}
\partial Z(\psi)=\sum_{k=1}^m \psi(x_k)\eta(x_k)\,.
\end{equation}
\end{itemize}
Under these conditions, we will often denote $Z=\llbracket\Sigma,\tau,\theta\rrbracket$, $\partial Z=\sum_{k=1}^m\eta(x_k)\delta_{x_k}$ and $\mathds M^{(n,\alpha)}(\partial Z):=\sum_{k=1}^m\|\eta(x_k)\|_\alpha$.  Moreover, we
write $Z=(Z_1,\ldots,Z_n)$, and we say
that $Z_j$ are the \emph{components} of $Z$, if $Z_j=\llbracket\Sigma_j,\tau_j,\theta_j\rrbracket$
$(j=1,\ldots,n)$ are classical\footnote{Throughout the paper we refer to ``classical'' currents to indicate currents with coefficients in $\R$ or $\Z$ (that is, $n=1$). An \emph{integral} 1-current is a classical rectifiable current with integer coefficients. We refer to \cite{Fe1}, \cite{Si} or \cite{Krantz_Parks} for the classical theory of currents.} integral currents satisfying
\[
Z(\omega)=\sum_{j=1}^n\int_{\Sigma_j} \langle\omega_j(x);\tau_j(x)\rangle\theta_j(x)\,d\Haus{1}(x)\,,
\]
for every $\omega\in\mathscr C^\infty_c(\R^d;\Lambda^1_{(n,\alpha)}(\R^d))$.

The appropriate notion of weak*-topology for the set of rectifiable currents with coefficients in $(\Z^n,\|\cdot\|_\alpha)$ is dual to the suitable locally convex topology on the space
$\mathscr C^\infty_c(\R^d;\Lambda^1_{(n,\alpha)}(\R^d))$, built in analogy with the topology on
$\mathscr C^\infty_c(\R^d)$ with respect to which distributions are dual. We write $Z_{i}\stackrel{*}{\rightharpoonup} Z$ when the sequence $(Z_i)_{i\ge 1}$ is weakly*-converging to $Z$.
}\end{definition}
Let us remark that these definitions are consistent with the classical theory of currents and currents with coefficients in a (normed, abelian, discrete) group. In particular, rectifiable currents with coefficients in $(\Z^n,\|\cdot\|_\alpha)$ belong to the larger linear space of currents with coefficients in $(\R^n,\|\cdot\|_\alpha)$ and the mass
of a current $Z$ is the supremum of $|Z(\omega)|$ when $\omega\in\mathscr C^\infty_c(U;\Lambda^1_{(n,\alpha)}(\R^d))$ has norm 
$$\sup_{x\in U}\|\omega(x)\|^*_\alpha\le 1.$$

Despite the complexity of the notation, this framework is not redundant for the objects we would like to represent and analyze. Let us clarify, firstly, how we endow a curve $\gamma$ with a canonical structure of current in the following example.

\begin{example}{\rm We associate to a Lipschitz path $\gamma:[0,1]\to\R^2$ (pa\-ra\-me\-tri\-zed with constant speed), and a coefficient
$\theta\in\R^n$, the 1-dimensional rectifiable current $Z=\llbracket\Gamma,\tau,\theta\rrbracket$ with coefficients in $\R^n$, where $\Gamma$ is the support of the curve $\gamma([0,1])$ and, denoting by  $\ell(\Gamma)$
the length of the curve, the orientation $\tau$ is defined by $\tau(\gamma(t)):=\gamma'(t)/\ell(\Gamma)$ for a.e. $t\in[0,1]$. The boundary of such current
is $\partial T=\theta(\delta_{\gamma(1)}-\delta_{\gamma(0)})$.}\end{example}

The next theorem provides a reasonable structure for a rectifiable
current with coefficients in $(\Z^n,\|\cdot\|_\alpha)$ as a countable sum of loops plus finitely many Lipschitz curves in $\R^d$ with constant multiplicity.

\begin{theorem}\label{thm:stru}
Let $Z$ be a rectifiable current in $\R^d$ with coefficients in ${(\Z^n,\|\cdot\|_\alpha)}$, then
\[
Z=\sum_{i=1}^m \tilde Z_i+\sum_{\ell=1}^\infty \mathring{Z}_\ell.
\]
Here $\tilde Z_i=\llbracket\tilde\Gamma_i,\tilde\tau_i,\tilde\theta_i\rrbracket$ are rectifiable currents with coefficients in $(\Z^n,\|\cdot\|_\alpha)$, $\tilde\Gamma_i$ being the support of an injective
Lipschitz
curve and $\tilde\theta_i$ a constant on $\tilde\Gamma_i$. Similarly $\mathring{Z}_\ell=\llbracket\mathring\Gamma_\ell,\mathring\tau_\ell,\mathring\theta_\ell\rrbracket$ are rectifiable currents with
coefficients in
$(\Z^n,\|\cdot\|_\alpha)$, $\mathring\Gamma_\ell$ being the support of a Lipschitz closed curve $\gamma_\ell:[0,1]\to\R^d$, which is injective on $(0,1)$. Again, $\mathring\theta_i$ is constant on $\mathring\Gamma_i$.
\end{theorem}
For the proof of this theorem and a more detailed statement see 2.3 in \cite{Con_Gar_Mas} (the proof rivisits the argument employed for classical integral currents in 4.2.25 of \cite{Fe1}).

\subsection{Compactness}
We recall the fundamental compactness result for rectifiable currents with coefficients in a group.

\begin{theorem}\label{thm:cptness}
Consider a sequence $(Z_i)_{i\ge 1}$ of rectifiable currents with coefficients in $(\Z^n,\|\cdot\|_\alpha)$ such that
\[
\sup_{i\ge 1}\mathds M^{(n,\alpha)}(Z_i)+\mathds M^{(n,\alpha)}(\partial Z_i)<+\infty\,.
\]
Then there exists a rectifiable current $Z$, with coefficients in $(\Z^n,\|\cdot\|_\alpha)$, and a subsequence $\left(Z_{i_h}\right)_{h\ge 1}$ such that
\[
Z_{i_h}\stackrel{*}{\rightharpoonup} Z\,.
\]
\end{theorem}
The proof of this theorem can be carried out componentwise, using the a\-na\-lo\-gous celebrated result for classical integral currents (see \cite{Fe1} or \cite{Am_Kirch}). Since the mass is lower semicontinuous with respect to the same topology, we directly get the existence of a mass-minimizing rectifiable current for a given boundary.

\begin{corollary}\label{cor:dirmeth}
Let $\hat Z$ be a rectifiable current with coefficients in ${(\Z^n,\|\cdot\|_\alpha)}$, then there exists a rectifiable current $\overline Z$ such that
\[
\mathds M^{(n,\alpha)}(\overline Z)=\min_{\partial Z=\partial\hat Z}\mathds M^{(n,\alpha)}(Z)\,,
\]
where the minimum is computed among rectifiable currents with coefficients in $(\Z^n,\|\cdot\|_\alpha)$.
\end{corollary}

\subsection{Calibrations}
The main reason why we want to treat the transportation problem as a mass minimization problem for rectifiable currents is the availability of the calibration technique as a tool to prove optimality.
\begin{definition}\label{def:calib}{\rm
Consider a rectifiable current $Z=\llbracket\Sigma,\tau,\theta\rrbracket$ with coefficients in $(\Z^n,\|\cdot\|_\alpha)$. A smooth $\R^n$-valued differential
form $\omega:\R^d\to \Lambda^1_{(n,\alpha)}(\R^d)$ is a \emph{calibration} for $Z$ if the following conditions hold:
\begin{itemize}
\item[(i)] for every $x\in\Sigma$ we have that $\langle\omega(x);\tau(x),\theta(x)\rangle=\|\theta(x)\|_\alpha$;
\item[(ii)] the form is closed\footnote{For the sake of readability, in Definition \ref{def:de} we did not define what the differential of a $1$-dimensional $\R^n$-valued form is.
On the one hand, all the examples we will provide have a constant calibration, so its closedness is trivial. On the other hand, it is not difficult to extend the classical definition of differential to $\R^n$-valued
$k$-forms componentwise.}, that is, ${\rm d}\omega=0$;
\item[(iii)] for every $x\in\R^d$, every unit vector $\tau\in\Lambda_1(\R^d)$ and every $h\in\R^n$ we have that
\[
\langle\omega(x);\tau,h\rangle\le\|h\|_\alpha\,.
\]
\end{itemize}
}\end{definition}
\begin{theorem}\label{thm_calib}
Let $Z=\llbracket\Sigma,\tau,\theta\rrbracket$ be a calibrated rectifiable current with coefficients in $(\Z^n,\|\cdot\|_\alpha)$, then $Z$ minimizes the mass among the rectifiable currents with coefficients
in $(\Z^n,\|\cdot\|_\alpha)$ with boundary $\partial Z$.
\end{theorem}
\begin{proof}
Let $\omega$ be a calibration for $Z$. Consider a competitor $Z'=\llbracket\Sigma',\tau',\theta'\rrbracket$ with boundary $\partial Z'=\partial Z$, then, as a consequence of Stokes' Theorem and
Definition \ref{def:calib}, we have that
\begin{eqnarray*}
\mathds M^{(n,\alpha)}(Z)&=&\int_\Sigma\|\theta(x)\|_\alpha\,d\Haus{1}(x)\\
&\stackrel{{\rm (i)}}{=}&\int_\Sigma\langle\omega(x);\tau(x),\theta(x)\rangle\,d\Haus{1}(x)\\
&\stackrel{{\rm (ii)}}{=}&\int_{\Sigma'}\langle\omega(x);\tau'(x),\theta'(x)\rangle\,d\Haus{1}(x)\\
&\stackrel{{\rm (iii)}}{\le}&\int_{\Sigma'}\|\theta'(x)\|_\alpha\,d\Haus{1}(x)=\mathds M^{(n,\alpha)}(Z')\,,
\end{eqnarray*}
where each equality (resp. inequality) is motivated by the corresponding relation in Definition \ref{def:calib}.
\end{proof}

\begin{remark}\label{calib_rect}
 In \S \ref{sec:tree}, we want to allow the multiplicity of the competitors to take values in $\R^n$. In other words $Z'$ could be a rectifiable current with coefficients in $(\R^n,\|\cdot\|_\alpha)$. The
existence of a calibration
for $Z$ guarantees the optimality of the mass of $Z$ even in this larger class. The proof of this fact is the same as that given above.
\end{remark}

\section{The Gilbert-Steiner problem}\label{sec:GS}
\subsection{The energy minimization problem}\label{em}
Let $X=(x_1,\ldots,x_m)$ and $Y=(y_1,\ldots,y_m)\in\left(\R^d\right)^m$, with $x_i\neq y_j$ for every $i,j=1,\ldots,m$. Denote by $B^{(1,1)}_{X,Y}$ the integral $0$-current
\begin{equation}\label{defbordo}
B^{(1,1)}_{X,Y}:=\sum_{i=1}^m(\delta_{y_i}-\delta_{x_i}).
\end{equation}
Consider the following problem.
\begin{itemize}
\item[(P1)] Among all (classical) integral $1$-currents $T=\llbracket\Sigma,\tau,\theta\rrbracket$ in $\R^d$ with $\partial T=B^{(1,1)}_{X,Y}$, find one which minimizes the \emph{Gilbert-Steiner energy}
\[\mathds E^\alpha(T)=\int_\Sigma|\theta(x)|^\alpha{\rm{d}}\cH^1(x)\,.\]
\end{itemize}
The problem (P1) is better known as the \emph{Gilbert-Steiner problem}\footnote{The actual version of the Gilbert-Steiner problem allows as competitors also polyhedral chains with real multiplicities. The issue whether or not this version is in general equivalent to our problem (P1) will be partially discussed later in \S\ref{rmk:conv}. For the purposes of this paper, such issue is not fundamental, indeed the method we use to prove that a certain object is a solution of problem (P1) is strong enough to prove also the optimality
among polyhedral chains with real multiplicities (see Remark \ref{calib_rect}).}, that is, a transportation problem with sources in $x_1,\ldots,x_m$ and wells in $y_1,\ldots,y_m$. The cost of the transportation
network is precisely the energy $\mathds E^\alpha$.  Note that it is possible that $x_i=x_j$ (respectively: $y_i=y_j$) for some $i\neq j$: indeed a point with higher multiplicity represents a source generating (respectively: a well absorbing) a more intense flow.

\begin{remark}{\rm
The existence of a solution for (P1) is again a consequence of the direct method, indeed \begin{itemize}
\item[(i)] the class of competitors is non-empty, in fact the current $B^{(1,1)}_{X,Y}$ is the boundary of a current with finite energy. Take, for instance, the integral $1$-current
$\hat T=\sum_{i=1}^m\llbracket\sigma_i,\tau_i,1\rrbracket$,
where $\sigma_i$ is the segment between $x_i$ and $y_i$ and $\tau_i=(y_i-x_i)/|y_i-x_i|$, then it is trivial to check that $\partial\hat T=B^{(1,1)}_{X,Y}$ and $\mathds E^\alpha(\hat T)<+\infty$;
\item[(ii)] the set of integral currents in $\R^d$ with boundary $B^{(1,1)}_{X,Y}$ and energy bounded by a constant is contained in a weak* compact set;
\item[(iii)] the energy $\mathds E^\alpha$ is lower semicontinuous with respect to the weak* to\-po\-lo\-gy.
\end{itemize}
}\end{remark}

\subsection{The structure of integral 1-currents}
Let us recall the analogue of Theorem \ref{thm:stru} in the simpler setting of $1$-dimensional integral
currents in $\R^d$, with the double aim of an easier picture and a better understanding of the Gilbert-Steiner problem.
\begin{theorem}\label{thm:struct_1currents}
Let $T$ be an integral $1$-current in $\R^d$, with\footnote{We will denote by $\mathds{M}$ the mass of a classical currents.} $\mathds{M}(\partial T)=2m$. Then
\begin{equation}\label{decomp_1current}
T=\sum_{i=1}^m \tilde T_i +\sum_{\ell=1}^{\infty} \mathring{T}_\ell\,,
\end{equation}
where $\tilde T_i$ are integral $1$-currents associated to injective Lipschitz paths, for every $i=1,\ldots,m$ and $\mathring{T}_\ell$ are integral $1$-currents associated to Lipschitz paths
$\gamma_\ell:[0,1]\to\R^d$
satisfying $\gamma_\ell(0)=\gamma_\ell(1)$
and injective on $(0,1)$, for every $\ell\ge 1$. In particular $\partial \mathring{T}_\ell=0$ for every $\ell\ge 1$. Moreover
\begin{equation}\label{spezza_massa}
\mathds M(T)=\sum_{i=1}^m \mathds M(\tilde T_i)+\sum_{\ell=1}^{\infty}\mathds M(\mathring{T}_\ell).
\end{equation}
\end{theorem}
\begin{remark}\label{rmk_permut}
Since $\mathds{M}(\partial T)=2m$ we can represent $\partial T=\sum_{i=1}^m(\delta_{y_i}-\delta_{x_i})$, where $(x_1,\ldots,x_m)$ and $(y_1,\ldots,y_m)\in\left(\R^d\right)^m$, satisfy $x_i\neq y_j$ for
every $i,j=1,\ldots,m$.
Without loss of generality, we may assume that there exists a permutation $\sigma\in{\mathcal S}_m$ such that in the decomposition \eqref{decomp_1current} we have
$\partial\tilde T_i=\delta_{y_i}-\delta_{x_{\sigma(i)}}$.
\end{remark}

\begin{remark}\label{rmk:ordec}
In Theorem \ref{thm:struct_1currents}, if we represent $\tilde T_i=\llbracket\tilde\Gamma_i,\tilde\tau_i,\tilde\theta_i\rrbracket$, we may assume that the multiplicity $\tilde\theta_i$ is always non negative,
hence
if $\cH^1(\tilde\Gamma_i\cap\tilde\Gamma_j)>0$, then by \eqref{spezza_massa} $\tilde\tau_i=\tilde\tau_j$ $\cH^1$-a.e. on $\tilde\Gamma_i\cap\tilde\Gamma_j$. This implies that the multiplicity of the
current $\tilde T_1+\ldots+\tilde T_m$
is $\tilde\theta_1+\ldots+\tilde\theta_m$.
\end{remark}

\section{Gilbert-Steiner problem and currents with coefficients in $\R^n$}\label{sec:equiv}
\subsection{The mass minimization problem}
In this section we show that the Gilbert-Steiner problem (P1) can be rephrased as a mass minimization problem for rectifiable currents with coefficients in $(\Z^n,\|\cdot\|_\alpha)$, where $n$ is the number of sources (resp. wells),
counted with multiplicity, in the boundary datum $X=(x_1,\ldots,x_n), Y=(y_1,\ldots,y_n)\in(\R^d)^n$.
\begin{itemize}
\item[(P2)] Among all permutations $\sigma\in{\mathcal S}_n$ and among all rectifiable currents $Z=\llbracket\Sigma,\tau,\theta\rrbracket$ in $\R^d$ with coefficients in $(\Z^n,\|\cdot\|_\alpha)$ and boundary
\[
B^{(n,\alpha)}_{\sigma(X),Y}=\sum_{j=1}^ng_j(\delta_{y_j}-\delta_{\sigma(x_j)})\,,
\]
find one which minimizes the mass
\[
\mathds M^{(n,\alpha)}(Z)=\int_\Sigma\|\theta(x)\|_\alpha{\rm{d}}\cH^1(x)\,.
\]
\end{itemize}
It will be clear from Definition \ref{def:pass} that fixing the permutation $\sigma$ in (P2) corresponds somehow to prescribe which source goes into which well (see also the description of the \emph{who goes where} problem given in \cite{Ber_Ca_Mo2}). This is not prescribed, instead, when we fix the boundary datum in (P1).
This is the reason
why in (P2) we need to minimize also among all permutations $\sigma\in{\mathcal S}_n$. Of course we may drop this condition when the set of sources contains only one point with multiplicity $n$, as in the case of the irrigation problem.

\subsection{The equivalence of (P1) and (P2)}

In the next theorem we state the equivalence between problems (P1) and (P2). Let us first describe a canonical way to associate
to a competitor for the problem (P1) a competitor for the problem (P2) and vice versa.

\begin{definition}\label{def:pass}{\rm
Let $X=(x_1,\ldots,x_n), Y=(y_1,\ldots,y_n)\in(\R^d)^n$, with $x_i\neq y_j$ for every $i,j=1,\ldots,n$.
\begin{itemize}
\item[(1)]
Given an integral current $T$ in $\R^d$ with boundary $\partial T=B^{(1,1)}_{X,Y}$ as in \eqref{defbordo}, consider a decomposition according to Theorem \ref{thm:struct_1currents}
(indeed, $\mathds M(\partial T)=2n$), satisfying the assumption of Remark \ref{rmk_permut}, that is
\[
T=\sum_{i=1}^n\tilde T_i+\sum_{\ell=1}^\infty \mathring T_\ell\,
\]
and there exists $\sigma\in{\mathcal S}_n$ such that
\[
\partial\tilde T_i=\delta_{y_i}-\delta_{x_{\sigma(i)}}\,.
\]
We set
\[
Z[T]:=(\tilde T_1,\ldots,\tilde T_n)\,,
\]
(according to the basis $(g_1,\ldots,g_n)$) which is a rectifiable current with coefficients in $(\Z^n,\|\cdot\|_\alpha)$ and boundary
\[
\partial (Z[T])=B_{\sigma(X),Y}^{(n,\alpha)}\,.
\]
Notice that $Z[T]$ depends on the decomposition chosen for $T$, therefore it may be non-unique.

\item[(2)]Vice versa, given a rectifiable current $Z$ with coefficients in $(\Z^n,\|\cdot\|_\alpha)$, we can write it componentwise as $Z=(Z_1,\ldots,Z_n)$ and we set
\[
T[Z]:=Z_1+\ldots+Z_n\,,
\]
which is an integral current with boundary $B_{X,Y}^{(1,1)}$.
\end{itemize}
}\end{definition}
\begin{remark}\label{rmk:cone}{\rm
Consider a rectifiable current $Z=\llbracket\Sigma,\tau,\theta\rrbracket$ with coefficients in $(\Z^n,\|\cdot\|_\alpha)$. Then
\[
\mathds E^\alpha(T[Z])\le\mathds M^{n,\alpha}(Z)\,.
\]
Indeed, the integral current $T[Z]$ has multiplicity $\theta_1+\ldots+\theta_n$, where $\theta=\theta_1g_1+\ldots+\theta_ng_n$ is the multiplicity of $Z$. Moreover,
\begin{equation}\label{eqqq}
\Big|\sum_{j=1}^n\theta_j\Big|^\alpha\le\Big(\sum_{j=1}^n|\theta_j|^\frac 1\alpha\Big)^\alpha=\|\theta\|_\alpha\,,
\end{equation}
because $\theta_j\in\Z$ for any $j=1,\ldots,n$ and $0<\alpha<1$. Notice that the equality holds in \eqref{eqqq} if and only if $\theta_j\in\{0,1\}$ for every $j=1,\ldots,n$.
}\end{remark}
\begin{proposition}\label{prop:coeff}
Consider an integral current $T$ in $\R^d$ with $\mathds M(\partial T)=2n$ and a decomposition of type \eqref{decomp_1current} as in Theorem \ref{thm:struct_1currents}. Then,
\[
\mathds M^{n,\alpha}(Z[T])=\mathds E^\alpha\Big(\sum_{i=1}^n\tilde T_i\Big)\,.
\]
\end{proposition}
\begin{proof}
Let us denote by $\theta=\tilde\theta_1g_1+\ldots+\tilde\theta_ng_n$ the multiplicity of $Z[T]$. By definition of $Z[T]$, each $\tilde\theta_j$ is the multiplicity of
$\tilde T_j=\llbracket\tilde\Gamma_j,\tilde\tau_j,\tilde\theta_j\rrbracket$. Since $\mathds M(\partial T)=2n$, then $\tilde\theta_j\in\{0,1\}$ almost everywhere on $\bigcup_j\tilde\Gamma_j$ (in particular,
$\tilde\theta_j=1$ a.e. on $\tilde\Gamma_j$).
Thus, we can infer that
\[
\|\theta\|_\alpha=|\tilde\theta_1+\ldots+\tilde\theta_n|^\alpha
\]
almost everywhere on $\bigcup_j\tilde\Gamma_j$. The conclusion follows, because the multiplicity of $\sum_{j=1}^n\tilde T_j$ is precisely $\tilde\theta_1+\ldots+\tilde\theta_n$ (see Remark \ref{rmk:ordec}).
\end{proof}

\begin{theorem}\label{equiva}
Let $X$ and $Y$ be as in Definition \ref{def:pass}. Then problems {\rm (P1)} and {\rm (P2)} are equivalent. More precisely, the following results hold.
\begin{itemize}
\item[(1)] If $T$ is an energy minimizer for the Gilbert-Steiner problem {\rm (P1)} with boundary $B_{X,Y}^{(1,1)}=\sum_{j=1}^n(\delta_{y_j}-\delta_{x_j})$, then $Z[T]$ is a mass minimizer for the problem
{\rm (P2)}. Moreover $\mathds E^\alpha(T)=\mathds M^{n,\alpha}(Z[T])$. 
\item[(2)] If $Z$ is a solution for problem {\rm (P2)} with the datum $(X,Y)$, then $T[Z]$ minimizes the energy
$\mathds E^\alpha$ among all integral currents with boundary $B_{X,Y}^{(1,1)}$. Moreover $\mathds M^{n,\alpha}(Z)=\mathds E^\alpha(T[Z])$.
\end{itemize}
\end{theorem}
\begin{proof}
\begin{itemize}
\item[(1)] Let $T$ be a solution of the Gilbert-Steiner problem (P1). Since we want to show that $Z[T]$ is a solution for the problem (P2), then we compare its mass with the mass of an admissible competitor for (P2), that is, 
a rectifiable current $Z'$ with coefficients in $(\Z^n,\|\cdot\|_\alpha)$ and boundary $B_{\sigma'(X),Y}^{(n,\alpha)}$ for some $\sigma'\in{\mathcal S}_n$. Then, we have that
\[
\mathds M^{n,\alpha}(Z')\stackrel{\ref{rmk:cone}}{\ge}\mathds E^\alpha(T[Z'])\ge\mathds E^\alpha(T)\stackrel{\eqref{spezza_massa}}{\ge}\mathds E^\alpha\Big(\sum_{i=1}^n\tilde T_i\Big)\stackrel{\ref{prop:coeff}}{=}\mathds M^{n,\alpha}(Z[T])\,.
\]
Let us explain that the first inequality is due to Remark \ref{rmk:cone}, the second one holds because $T$ is a minimizer for the energy $\mathds E^\alpha$ with boundary $B_{X,Y}^{(1,1)}$
and the third one holds because of \eqref{spezza_massa}. Finally, the last equality has already been proved in Proposition \ref{prop:coeff}. Notice also that the third inequality must be in fact an equality,
because
$\sum_{i=1}^n\tilde T_i$ is a competitor of $T$ in (P1).

\item[(2)] Let $Z$ be a solution of the problem (P2). We claim that $T[Z]$ is a solution of (P1): indeed, if $T'$ is an integral current with boundary $B_{X,Y}^{(1,1)}$, then we have that
\[
\mathds E^\alpha(T')\stackrel{\eqref{spezza_massa}}{\ge}\mathds E^\alpha\Big(\sum_{i=1}^n\tilde T'_i\Big)\stackrel{\ref{prop:coeff}}{=}\mathds M^{n,\alpha}(Z[T'])\ge\mathds M^{n,\alpha}(Z)\stackrel{\ref{rmk:cone}}{\ge}\mathds E^\alpha(T[Z])\,.
\]
As before, let us remind that the first inequality is a consequence of \eqref{spezza_massa} and the equality is due to Proposition \ref{prop:coeff}. Finally, the last inequality relies on
Remark \ref{rmk:cone}. This last is an equality, indeed. This follows from the fact that the multiplicity of every component $Z_j=\llbracket\Sigma_j,\tau_j,\theta_j\rrbracket$ of $Z$ is 1
($\cH^1$-a.e. on $\Sigma_j$) and the inequality in \eqref{eqqq} is an equality when every $\theta_j$ is either 0 or 1.
\end{itemize}
\end{proof}

\subsection{Examples of calibreted minima}
We now show some examples of calibrations for the problem (P2) with the only aim to make the reader confident with this technique and to throw light on some issues that one may not notice immediately in the general formulation. In the first example there are two sources and two wells with unit multiplicity on the vertices of a rectangle in $\R^2$. The second example is an irrigation problem with only two wells, but where we allow also higher multiplicities.
\begin{example}\label{ex:calib}{\rm
Let $\alpha=1/2$ and consider the following points in the plane
\[
x_1:=(-3,-2),\;x_2:=(-2,-3),\;y_1:=(1,0),\;y_2:=(0,1),\;z:=(-2,-2)\,.
\]
Consider the 0-dimensional current in $\R^2$ with coefficients in $\Z^2$ given by
\[
B_{(x_1,x_2),(y_1,y_2)}^{(2,1/2)}=g_1(\delta_{y_1}-\delta_{x_1})+g_2(\delta_{y_2}-\delta_{x_2})\,.
\]

We claim that a solution of problem (P2) with boundary $B_{(x_1,x_2),(y_1,y_2)}^{(2,1/2)}$ is
\[
Z=\llbracket\sigma_1,\tau_1,g_1\rrbracket+\llbracket\sigma_2,\tau_2,g_2\rrbracket+\llbracket\sigma_3,\tau_3,g_1+g_2\rrbracket+\llbracket\sigma_4,\tau_1,g_1\rrbracket+\llbracket\sigma_5,\tau_2,g_2\rrbracket
\]
where $\sigma_1$ is the segment from $x_1$ to $z$ with orientation $\tau_1=\left(1,0\right)$, $\sigma_2$ is the segment from $x_2$ to $z$ with orientation $\tau_2=\left(0,1\right)$, $\sigma_3$ is the segment from $z$ to $0$ with orientation $\tau_3=\left(\sqrt{2}/2,\sqrt{2}/2\right)$, $\sigma_4$ is the segment from $0$ to $y_1$ and $\sigma_5$ is the segment from $0$ to $y_2$.

Firstly, we exhibit a constant\footnote{Since $\omega$ is constant, condition (ii) in Definition \ref{def:calib} is automatically satisfied.} calibration $\omega(x)\equiv\omega\in\Lambda^1_{(2,1/2)}(\R^2)$, which proves that $Z$ is a mass minimizer for the boundary $B_{(x_1,x_2),(y_1,y_2)}^{(2,1/2)}$. 
Then, we will prove that the achieved minimum (that is, the value $\mathds M^{(2,1/2)}(Z)$) is less or equal to the minimum for the boundary $B_{(x_1,x_2),(y_2,y_1)}^{(2,1/2)}$ (remember that in (P2) we require to minimize among all possible permutations). 

Let us remind that an $\R^2$-valued covector acts bilinearly
on the pair $(\tau,g)\in\Lambda_1(\R^2)\times\Z^2$  and we can represent it as a $2\times 2$ matrix with real coefficients. We let
\[
\omega(x)\equiv\left(
\begin{array}{cc}
1 & 0 \\
0 & 1
\end{array}
\right)\,.
\]
Now we can easily check that $\langle\omega;\tau_i,g_i\rangle=1$ for $i=1,2$, while
\[
\langle\omega;\tau_3,g_1+g_2\rangle=\sqrt 2=\|g_1+g_2\|_\alpha\,,
\]
so condition (i) in Definition \ref{def:calib} is fulfilled. Condition (iii), concerning the norm of $\omega$, is trivially verified (for $\alpha=1/2$ it is sufficient to check that the Euclidean norm of the matrix is less or equal to 1).

To prove that the minimum $\mathds M^{(2,1/2)}(Z)$ is less or equal to the minimum for the boundary $B_{(x_1,x_2),(y_2,y_1)}^{(2,1/2)}$, we fix a competitor $Z'$ for this second problem. Without loss of generality, we can consider $Z'$ as the sum of two curves, one with multiplicity $g_1$ and one with multiplicity $g_2$. There are two cases: either the two curves in $Z'$ do not intersect, in this case
\begin{equation}\label{sep}
\mathds{M}^{(2,1/2)}(Z')\ge|x_1-y_2|+|x_2-y_1|=6\sqrt 2>8=\mathds{M}^{(2,1/2)}(Z)\,,
\end{equation}
or the curves intersect at some point $z'$. In this case, it is sufficient to ``switch'' the multiplicity $g_1$ with $g_2$ from that point to the part of the boundary supported on $\{y_1,y_2\}$. 

\begin{figure}[htbp]
\begin{center}
\scalebox{1}{
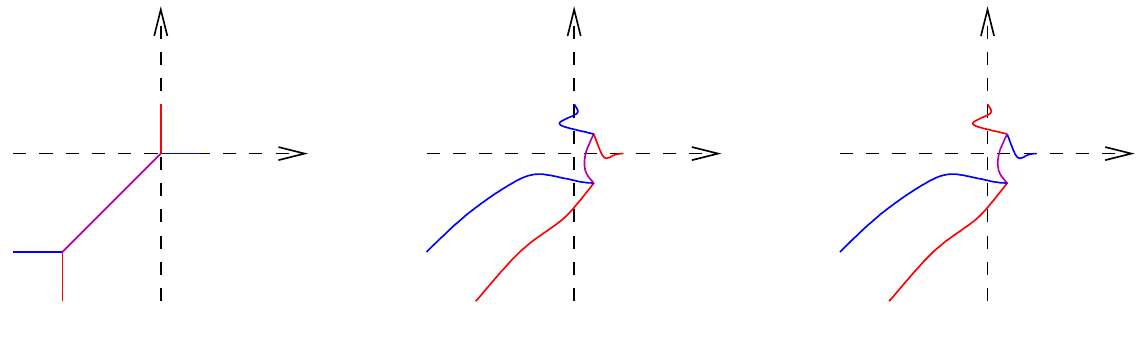
}
\end{center}
\end{figure}

In this way, we created a competitor $Z''$ for the problem with boundary $B_{(x_1,x_2),(y_1,y_2)}^{(2,1/2)}$ with the property that
\[
\mathds{M}^{(2,1/2)}(Z')=\mathds{M}^{(2,1/2)}(Z'')\,.
\]
Therefore we may conclude
\[
\mathds{M}^{(2,1/2)}(Z')=\mathds{M}^{(2,1/2)}(Z'')\ge\mathds{M}^{(2,1/2)}(Z)\,.
\]
}
\end{example}
\begin{example}{\rm
Fix $\alpha=1/2$ and $\theta:=\arccos(1/\sqrt{3})\in(0,\pi/2)$. Fix $x,y_1,y_2\in\R^2$ given by $x=(-\cos\theta,-\sin\theta)$, $y_1=(1,0)$ and $y_2=(0,1)$. Denote 
$$X=(x,x,x)\in(\R^2)^{3}$$ 
and 
$$Y=(y_1,y_2,y_2)\in(\R^2)^{3}.$$
We claim that a mass-minimizing current with coefficients in $(\Z^2,\|\cdot\|_{1/2})$ among those with boundary $B_{X,Y}^{(2,1/2)}$ is
\[
T=\llbracket\hat\sigma_0,\hat\tau_0,g_1+g_2+g_3\rrbracket+\llbracket\hat\sigma_1,\hat\tau_1,g_1\rrbracket+\llbracket\hat\sigma_2,\hat\tau_2,g_2+g_3\rrbracket
\]
where $\hat\sigma_0$ is the segment from $x$ to $0$ and $\hat\sigma_i$ is the segment from $0$ to $y_i$ for $i=1,2$. The orientation $\hat\tau_i$ on each segment goes from the first to the second extreme point.

\begin{figure}[htbp]
\begin{center}
\scalebox{1}{
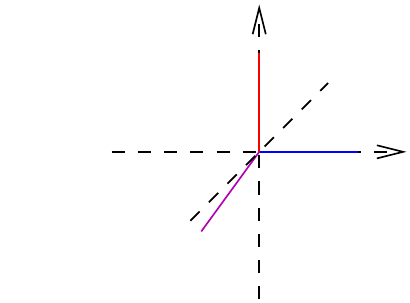
}
\end{center}
\end{figure}

The constant $\R^3$-valued form 
\[
\hat\omega(x)\equiv\left(\begin{array}{rr}
1 & 0 \\
0 & \sqrt{2}/2\\
0 & \sqrt{2}/2
\end{array}\right)
\]
is a calibration for $T$. Indeed 
$$\langle\hat\omega;\hat\tau_0,g_1+g_2+g_3\rangle=\sqrt{3}=\|g_1+g_2+g_3\|_{1/2};$$
$$\langle\hat\omega;\hat\tau_1,g_1\rangle=1=\|g_1\|_{1/2};$$
$$\langle\hat\omega;\hat\tau_2,g_2+g_3\rangle=\sqrt{2}=\|g_2+g_3\|_{1/2},$$
therefore condition (i) is satisfied. Moreover $\|\langle\hat\omega;\tau,\cdot\rangle\|_{1/2}=1$ for every $\tau\in\R^2$, hence (iii) is also satisfied. As before, condition (ii) is trivially verified because $\hat\omega$ is constant.

}\end{example}

\subsection{Rescaled problems}\label{subsec:resc}

In \S \ref{sec:tree} we want to deal with a boundary made of infinitely many points. To this aim it is worthwhile mentioning the approach by Xia in \cite{Xia1}, where the continuous version of the Gilbert-Steiner problem (i.e. $\mu_-$ and $\mu_+$ are not necessarily atomic measures) is defined through the relaxation of the Gilbert-Steiner energy. Given 
a normal 1-dimensional current $T$ one considers all possible sequences $T_n$ of polyedral currents which are approximating $T$ in the flat norm, then one defines the energy of $T$ as
$$\mathds E^\alpha(T):=\inf_{T_n}\{\liminf_{n}\mathds E^\alpha(T_n)\}.$$
As shown in Theorem 2.7 of \cite{Xia2}, only rectifiable currents enter in the competition, since non-rectifiable normal currents have infinite energy. 

As a first step to prove the optimality of our irrigation network in \S\ref{sec:tree}, it is convenient to introduce a renormalized version of the problem (P1), in which the boundary is
rescaled to have mass equal to 2.
More precisely, consider the following problem, where $B_{X,Y}^{(1,1)}$ and $n$ are as in \S\ref{em}.
\begin{itemize}
\item[(P1')] Among all rectifiable $1$-currents $T'=\llbracket\Sigma,\tau,\theta\rrbracket$ in $\R^d$ with real coefficients, with $T'=n^{-1}T$ and $T$ is an integral current with $\partial T=B_{X,Y}^{(1,1)}$,
find one which minimizes the extended Gilbert-Steiner energy
\[
\mathds E^\alpha(T')=\int_\Sigma\|\theta\|^\alpha\,d\Haus{1}\,.
\]
\end{itemize}
There is a clear bijection between the competitors of problems (P1) and (P1'). The feature we want to underline is how the energy rescales. Indeed 
\begin{equation}\label{e1}
\mathds E^\alpha(n^{-1}T)=n^{-\alpha}\mathds E^\alpha(T),
\end{equation}
for every $T$ which is a competitor for (P1).

The quantity we minimize in the problem (P2), instead, has naturally a linear rescaling, therefore if we want to maintain the equality between the minima, we have to rescale the problem (P2) differently. It turns out
that the correct version of the rescaled problem is the following.
\begin{itemize}
\item[(P2')] Among all permutations $\sigma\in{\mathcal S}_n$ and among all the rectifiable currents $Z'=\llbracket\Sigma,\tau,\theta\rrbracket$ in $\R^d$ with coefficients in $n^{-\alpha}\Z^n$ and with boundary
\[
\partial Z'=n^{-\alpha} B_{\sigma(X),Y}^{(n,\alpha)}=n^{-\alpha}\sum_{j=1}^n g_j(\delta_{y_j}-\delta_{x_{\sigma(j)}})\,,
\]
find one which minimizes the mass
\[
\mathds{M}^{n,\alpha}(Z')=\int_\Sigma\|\theta\|_\alpha\,{\rm{d}}\cH^1.
\]
\end{itemize}
Note that, if $Z$ is a competitor for (P2), then $n^{-\alpha}Z$ is a competitor for (P2') and naturally 

\begin{equation}\label{e2}
\mathds{M}^{n,\alpha}(n^{-\alpha}Z)=n^{-\alpha}\mathds{M}^{n,\alpha}(Z).
\end{equation}
The canonical way to associate to a competitor for the problem {\rm (P1')} a competitor for the problem {\rm (P2')} and vice versa is a
straightforward modification of the method explained in Definition \ref{def:pass}.
\begin{itemize}
\item[(1)] To a competitor $T'$ for {\rm (P1')} we associate the current
$$
Z'[T']:=n^{-\alpha}Z[nT']
$$
(note that $nT$ is a competitor for {\rm (P1)}).
\item[(2)] To a competitor $Z'$ for {\rm (P2')} we associate the current
$$
T'[Z']:=n^{-1}T[n^\alpha Z']\,.
$$
\end{itemize}
It is easy to observe that with this rescaling the analog of Remark \ref{rmk:cone} and of Proposition \ref{prop:coeff} still hold, therefore we have the following corollary.
\begin{corollary}\label{coroll}
Problems {\rm (P1')} and {\rm (P2')} are equivalent in the sense of Theorem \ref{equiva}.
\end{corollary}
\begin{proof}
By construction, the solutions of the problem (P1') are in bijection with the solutions of the problem (P1). Analogously the solutions of the problem (P2') are in bijection with the solutions of the problem (P2).
Thus the equivalence follows from Theorem \ref{equiva} and from \eqref{e1} and \eqref{e2}.
\end{proof}

\subsection{Convex problems and existence of calibrations}\label{rmk:conv}
Even if we have found a description of the Gilbert-Steiner problem which consists in a minimization of a convex functional, we cannot really claim that we are solving a ``convex problem'' (in which case a 
powerful machinery for the computation of minima would be available). Indeed our class of competitors is not a convex set. While for the classical formulation of the Gilbert-Steiner problem (see footnote 3) one can safely extend 
the minimization problem to a convex and compact class of objects (being the energy infinite for non-rectifiable currents), this is not the case for problem (P2). Moreover, in (P2) we cannot guarantee that there is no gap of mass between the minimizers among normal currents and minimizers among integral currents. An interesting occurrence of this gap phenomenon in a 
non-Euclidean setting is given in \cite{AnMa}. We conjecture that in the Euclidean setting there is no such gap. The validity of this conjecture would imply that 
the Gilbert-Steiner problem can be viewed as a convex problem and moreover the existence of (a sort of) calibration would be a necessary and sufficient condition for the minimality (see \S 4 of \cite{AnMa} for more details).

\section{A minimizer with infinitely many branching points}\label{sec:tree}
\subsection{Currents in metric spaces}
In this section we prove the optimality of a certain irrigation network having an infinite boundary datum, as mentioned in \S\ref{subsec:resc}. The main tools to prove
its optimality have already been introduced in \S \ref{sec:not} and \S \ref{sec:equiv}, more precisely in Theorem \ref{thm_calib} (the existence of a calibration is a sufficient condition for minimality) and in Corollary \ref{coroll} (the equivalence of problems (P1') and (P2')). Nevertheless, since the ambient space for this
irrigation
network is the separable Hilbert space\footnote{We choose this setting instead of the more natural Euclidean space, because the possibility to have infinitely many mutually orthogonal direction allows us to find a very simple calibration for the problem we will introduce. } $\ell^2$, a short digression on currents in metric spaces is needed. In order to keep the discussion here as incisive as possible, we recall only those facts
that are strictly unavoidable to define a certain irrigation problem and to prove the optimality of our configuration. The reader is referred to \cite{Am_Kirch} for the theory of
currents in a metric space and for all the relevant definitions. Throughout this section, we will assume $\alpha=1/2$.

Let $T$ be a rectifiable 1-current in $\ell^2$ (see \S 4 in \cite{Am_Kirch}) and let $\|T\|$ be the associated \emph{mass} (in the sense of Definition 2.6 of \cite{Am_Kirch}). In particular, $T$ is supported on a countably $\cH^1$-rectifiable
set $\Sigma$ and $\|T\|=\theta\cH^1\res\Sigma$ for some non-negative $\theta\in L^1(\Sigma,\R)$. We call the \emph{extended Gilbert-Steiner energy}\footnote{We denote the extended Gilbert-Steiner energy by $\mathds E$, omitting $\alpha$, because $\alpha=1/2$ throughout all the current section.} of $T$ the quantity
\begin{equation}\label{ggse}
\mathds{E}(T):=\int_\Sigma \sqrt{\theta(x)}\,{\rm{d}}\cH^1(x)\,,
\end{equation}
if the integral makes sense, or $+\infty$ otherwise. One may notice immediately that the extended Gilbert-Steiner energy is still subadditive, i.e.,
\[
\mathds{E}(T_1+T_2)\le\mathds{E}(T_1)+\mathds{E}(T_2)\,.
\]
Since the projection on a finite dimensional space is going to be a fundamental simplification in the handling of our problem, we recall the following proposition, which just relies on the concavity of the square root. An analogous statement for the mass is given in equation (2.4) of \cite{Am_Kirch}.
\begin{proposition}\label{prop:proj}
Let $\pi:\ell^2\to V$ be the orthogonal projection of $\ell^2$ onto a subspace $V$. Then, for every rectifiable 1-current $T$ in $\ell^2$, there holds\footnote{For the definition of the pushforward $F_\sharp T$ of a current $T$ through a map $F$, see Definition 2.4 of \cite{Am_Kirch}.} $\mathds{E}(\pi_\sharp T)\le \mathds{E}(T)$.
\end{proposition}

\subsection{Construction of the irrigation tree}\label{subs:constr}

Let $(e_i)_{i\in\N}$ be the standard orthonormal basis of $\ell^2$. Roughly speaking, we are going to define a rectifiable $1$-dimensional current $T$ in $\ell^2$ (with coefficients in $\Q$) in the following way: we start from a ``trunk'', which is the oriented segment from the origin $O$ to $e_1\in\ell^2$, with unit multiplicity, then, for every $n\in\N$, we are going to build a family of $2^n$ ``branches'', that is, $2^n$ currents supported on segments connected to the main ``tree'', with suitable length, orientation and multiplicity in $\Q$. The irrigation tree will be the sum, up to infinity, of these currents.

To begin with, we deal with the orientations of the segments. If it is not obvious from the context, we always denote by $\sigma(p,q)$ a segment in $\ell^2$ from $p$ to $q$ and by $x_{\sigma(p,q)}$ the (unit) orientation of $\sigma(p,q)$ from $p$ to $q$. We now establish the directions $y(x),z(x)$ according to which we will assign the orientation of the segments of a new generation. For every $x\in\ell^2$ of the form $x=\sum_{i=1}^l a_ie_i$, with $a_l\neq 0$, we put
\[
y(x):=\frac{\sqrt 2}{2}\sum_{i=1}^l a_i(e_i+e_{l+i})
\] 
and
\[
z(x):=\frac{\sqrt 2}{2}\sum_{i=1}^l a_i(e_i-e_{l+i})\,.
\]
Notice that $y(x)$ and $z(x)$ are orthogonal and $y(x)+z(x)$ is parallel to $x$. Moreover, if we start with a set of  mutually orthonormal directions $\{x_1,\ldots,x_m\}$, then $\{y(x_1),z(x_1),\ldots,y(x_m),z(x_m)\}$ are mutually orthonormal, too.


We consider the following iterative construction.
\begin{enumerate}
\item[$\circ$] We set 
\[
T_0=\llbracket\sigma(O,e_1),x_0,1\rrbracket\,,
\] 
where $\sigma(O,e_1)$ is the segment from the origin to $e_1$ and $x_0=e_1$ is the outgoing orientation of $\sigma(O,e_1)$. Moreover, we define $E_0:=\{\sigma(O,e_1)\}$.
\item[$\circ$] For every line segment $\sigma(p,q)\in E_{n-1}$ with orientation $x=x_{\sigma(p,q)}$, we put in the set $E_n$ the pair of segments 
$$\sigma(q,q+4^{-n}y(x))\;\;\; {\rm and}\;\;\; \sigma(q,q+4^{-n}z(x)).$$ 
Finally, we define
\[
T_n:=\sum_{i\le n}\sum_{\sigma\in E_i}\llbracket\sigma,x_\sigma,2^{-i}\rrbracket\,.
\]
Notice that\footnote{Here and in the following, $\mathds M(T)$ denotes the total mass of the mass measure $\|T\|$.} $\mathds M(\partial T_n)=2$ for every $n\in\N$.
\end{enumerate}
If we call $$E^n:=\bigcup_{\sigma\in E_n}\sigma,$$ 
then the support of the irrigation tree is defined as the limit
\[
E:=\bigcup_{n\in\N}E^n\,.
\]
Notice that the segments are growing in orthogonal directions, hence $\cH^1$-a.e. $p\in E$ belongs to a unique segment $\sigma(p)\in E_{n(p)}$ and we can define an orientation $\tau:E\to\ell^2$ setting $\tau(p)=x_{\sigma(p)}$. Analogously, we can define a multiplicity $m:E\to\Q$ setting $m(p)=2^{-n}$ for every $p\in E^n$. The irrigation tree $T$ is defined as 
\[
T:=\llbracket E,\tau,m\rrbracket\,.
\]

The truncated currents $T_n$ play a fundamental role in the proof of optimality of the irrigation tree $T$, because they are a good approximation for $T$ (as we underline below in Remark \ref{rmk:approx_tn}) and they are essentially embedded in a finite dimensional space (indeed, ${\rm supp}\,T_n\subset{\rm span}\{e_1,\ldots,e_{2^n}\}$). Since the segments in $E$ are essentially disjoint, we can also write $T_n=T\res\left(\bigcup_{i\le n} E^n\right):=\llbracket \bigcup_{i\le n} E^n,\tau,m\rrbracket$.


\begin{remark}\label{rmk:approx_tn}{\rm By construction, the following facts hold.
\begin{itemize}
\item[(i)] $\mathds E(T-T_n)=\sum_{j>n}2^j2^{-\frac j2}4^{-j}=\sum_{j>n}2^{-\frac 32 j}$.
\item[(ii)] $\mathds M(T-T_n)=\sum_{j>n}2^j2^{-j}4^{-j}=\sum_{j>n} 2^{-2j}$.
\item[(iii)] If $\pi_n:\ell^2\to{\rm span}\{e_1,\ldots,e_{2^n}\}$ are the orthogonal projections and if $\{q_1,\ldots,q_{2^n}\}$ are the second extreme points of the segments in $E_n$, then $({\pi_n})\sharp(T_n)$ is
a candidate for the minimization problem (P1') in \ref{subsec:resc}, where
\[
X=(0,\ldots,0),\quad Y=(q_1,\ldots,q_{2^n})\,.
\]
Indeed $2^n({\pi_n})_\sharp T_n$ is an integral current.
\end{itemize}
}\end{remark}

We are now ready to prove the optimality of $T$.
\begin{theorem}\label{thm_tree}
The rectifiable $1$-current $T$ in $\ell^2$ minimizes the generalized Gilbert-Steiner energy \eqref{ggse} among all rectifiable currents in $\ell^2$ with boundary $\partial T$.
\end{theorem}

\begin{proof}
We split the proof into several steps.

\noindent{\bf Step 1: general strategy.} Let $S$ be an admissible competitor, i.e., $S$ is a rectifiable $1$-current in $\ell^2$ with $\partial S=\partial T$. 

We want to reduce the problem to an energy minimization problem for rectifiable $1$-currents in a finite dimensional space and then exploit the tools developed in the previous sections. For every $n\in\N$, let $S_n:=S-T+T_n$ and notice that $\partial S_n=\partial T_n$. By Proposition \ref{prop:proj}, we have that
\[
\mathds E(({\pi_n})_\sharp S_n)\le\mathds E(S_n)\quad\forall\,n\in\N\,.
\]

Fix $\varepsilon>0$. Since, by
subadditivity of $\mathds E$, we have that
\[
\mathds E(S_n)\le\mathds E(S)+\mathds E(T_n-T)\quad\forall\,n\in\N
\]
and the energy $\mathds E(T_n-T)$ is going to $0$ (see Remark \ref{rmk:approx_tn}), then there exists $n(\varepsilon)\in\N$ such that $\mathds E(T_{n(\varepsilon)}-T)\le\varepsilon$, thus $\mathds E(S)\ge \mathds E(S_{n(\varepsilon)})-\varepsilon$.
If we can prove that
\begin{equation}\label{claim1}
\mathds E(T_n)\le\mathds E(({\pi_n})_\sharp S_n)
\end{equation}
for every $n\in\N$, then we have that
\begin{multline*}
\mathds E(T)\le\mathds E(T_{n(\varepsilon)})+\mathds E(T_{n(\varepsilon)}-T)\le\mathds E(T_{n(\varepsilon)})+\varepsilon\\
\le\mathds E(({\pi_{n(\varepsilon)}})_\sharp S_{n(\varepsilon)})+\varepsilon\le\mathds E(S_{n(\varepsilon)})+\varepsilon\le\mathds E(S)+2\varepsilon\,,
\end{multline*}
which proves the theorem, since $\varepsilon$ can be arbitrarily small.

\noindent{\bf Step 2: reduction to problem (P1').} We have already noticed that $({\pi_n})_\sharp T_n$ is an admissible candidate for the problem (P1'). The issue is that, in general, $({\pi_n})_\sharp S_n$ is not an admissible competitor and we cannot immediately look for a calibration for the current $Z'[({\pi_n})_\sharp T_n]$\footnote{The definition of $Z'[({\pi_n})_\sharp T_n]$ is given in \S \ref{subsec:resc}. The role of the basis $(g_1,\ldots,g_{2^n})$ here is played by $(\pi_n(x_1),\ldots, \pi_n(x_{2^n}))$, where $x_1,\ldots,x_{2^n}$ are the (mutually orthogonal) orientations of the segments in the family $E_n$}. Indeed $({\pi_n})_\sharp S_n$ is a rectifiable current, with admissible boundary datum, but its coefficients are in $\R$ and not
necessarily in $2^{-n}\Z$, as required in (P1')\footnote{In particular, we do not have a canonical way to associate to $({\pi_n})_\sharp S_n$ a current with coefficients in $\R^{2^n}$ whose mass is less or equal to the energy of $({\pi_n})_\sharp S_n$ (see Proposition \ref{prop:coeff}).}.

In order to overcome this issue, we have to perform a second approximation. Fix $\delta>0$. We associate to $({\pi_n})_\sharp S_n$ a polyhedral current
$\hat S=\frac{1}{2^nN} R$, where $N\in\N$ and $R$ has integer multiplicity. We can choose $\hat S$ in such a way that the boundary is unchanged, i.e.,
\[
\partial\hat S=\partial({\pi_n})_\sharp S_n=\partial(\pi_n)_\sharp T_n\,,
\] 
and
\begin{equation}\label{ap}
\mathds E(\hat S)\le\mathds E\left(({\pi_n})_\sharp S_n\right)+\delta\,.
\end{equation}
The existence of $\hat S$ is a consequence of Proposition 4.4 of \cite{Xia2}.

\noindent {\bf Step 3: construction of $\hat Z$.} 

Since we want to move the problem of energy minimization in the setting of (rational multiples of) rectifiable currents with coefficients in $\Z^{2^n}$, we associate to $\hat S$ a current $\hat Z$ in $\R^{2^n}$ with coefficients in $\R^{2^n}$. We can choose $\hat Z$ in such a way that the boundary is 
\[
\partial\hat Z=\partial Z'[(\pi_n)_\sharp T_n]
\]
and
\begin{equation}\label{ineq2}
\mathds M^{(2^n,1/2)}(\hat Z)\le\mathds E(\hat S)\,.
\end{equation}
The construction of $\hat Z$, starting from $\hat S$, is analogous to that presented in Definition \ref{def:pass}. Let $x_1,\ldots, x_{2^n}$ be the (mutually orthogonal) orientations of the segments in the 
family $E_n$. Let us write the polyhedral integral 1-current $R=2^nN\hat S$ as
\[
R=\sum_{i\in I}\llbracket\sigma_i,\tau_i,a_i\rrbracket\,,\] 
where $\sigma_i$ is an oriented segment with orientation $\tau_i$ and $a_i\in\left\{1,\ldots,2^nN\right\}$. Decomposing\footnote{Without loss of generality we can assume $R$ to be acyclic, i.e., $\mathring{R}_\ell=0$ for every $\ell\in\N$.} $R$ as in \eqref{decomp_1current}, one may notice that the current $Z[R]$ (ac\-cor\-ding to the basis $(\pi_n(x_1),\ldots, \pi_n(x_{2^n}))$) has the form 
\[
Z[R]=\sum_{i\in I}\llbracket\sigma_i,\tau_i,b_i\rrbracket\,,
\]
where
\begin{enumerate}
\item[$\circ$] $b_i=c_{i,1} \pi_n(x_1)+\ldots+c_{i,2^n}\pi_n(x_{2^n})$;
\item[$\circ$] $c_{i,j}\in\N\cup\{0\}$ and $c_{i,j}\le N$ for every $j=1,\ldots,2^n$\footnote{Roughly speaking, $c_{i,j}$ is the portion of the multiplicity $a_i$ of $\sigma_i$, which is due to all paths ending in the second extreme of the segment of $E_n$ having orientation $x_j$};
\item[$\circ$] $c_{i,1}+\ldots+c_{i,2^n}=a_i$.
\end{enumerate}

We can compute
\[
\|b_i\|=\sqrt{c_{i,1}^2+\ldots+c_{i,2^n}^2}\le \sqrt{c_{i,1}N+\ldots+c_{i,2^n}N}=\sqrt{N a_i}\,.
\]
Denote $\hat Z:=2^{-n/2}N^{-1}Z[R]$. We have

\begin{eqnarray*}
\partial\hat Z &=&\partial(2^{-n/2}N^{-1}Z[R])=2^{-n/2}N^{-1}\partial Z[N2^n\hat S]\\
&=& 2^{-n/2}N^{-1}\partial Z[N2^n(\pi_n)_\sharp T_n]=\partial(2^{-n/2} Z[2^n(\pi_n)_\sharp T_n])=\partial Z'[(\pi_n)_\sharp T_n]\,.
\end{eqnarray*}
Moreover 
\begin{eqnarray*}
\mathds M^{(2^n,1/2)}(\hat Z)&=&2^{-n/2}N^{-1}\mathds M^{(2^n,1/2)}(Z[R])=2^{-n/2}N^{-1}\sum_{i\in I}\|b_i\|\Haus{1}(\sigma_i)\\
&\le & 2^{-n/2}N^{-1/2}\sum_{i\in I}\sqrt{a_i}\Haus{1}(\sigma_i)=2^{-n/2}N^{-1/2}\mathds E(R)\\
&=&\mathds E(2^{-n}N^{-1}R)=\mathds E(\hat S)\,.
\end{eqnarray*}

\noindent {\bf Step 4: proof of the minimality of $Z'[({\pi_n})_\sharp T_n]$.} 
Even if, technically speaking, $\hat Z$ does not belong to the set of competitors for (P2'), we can still prove that 
\begin{equation}\label{ineq3}
\mathds M^{(2^n,1/2)}\left(Z'[({\pi_n})_\sharp T_n]\right)\le\mathds M^{(2^n,1/2)}(\hat Z)
\end{equation}
via the calibration technique (see Remark \ref{calib_rect}).

Consider the basis of $\R^{2^n}$ $\left(\pi_n(x_1),\ldots,\pi_n(x_{2^n})\right)$. The current $Z'\left[({\pi_n})_\sharp T_n\right]$ has the following property: the multiplicity on each line
segment of its support is a multiple of the orientation of the segment itself\footnote{This is due to the fact that the sum of $y(x)$ and $z(x)$, defined in \S \ref{subs:constr}, is parallel to $x$.} (the orientation being defined above). This implies that, if we denote by $\omega:\R^{2^n}\to \R^{2^n}$ the identity map, then $\omega$ can
be seen as a $\R^{2^n}$-valued (constant) differential $1$-form, which in particular is a calibration for $Z'\left[({\pi_n})_\sharp T_n\right]$: the verification of conditions (i), (ii) and (iii) of Definition \ref{def:calib} is elementary, as in Example \ref{ex:calib}. Theorem \ref{thm_calib} and Remark \ref{calib_rect} allow us to conclude that \eqref{ineq3} holds.

\noindent {\bf Step 5: conclusion.} 
Eventually we can prove the claim \eqref{claim1}, indeed
\[
\mathds E(({\pi_n})_\sharp S_n)+\delta\stackrel{\eqref{ap}}{\ge}\mathds E(\hat S)\stackrel{\eqref{ineq2}}{\ge}\mathds M^{(2^n,1/2)}(\hat Z)\stackrel{\eqref{ineq3}}{\ge}\mathds M^{(2^n,1/2)}\left(Z'[({\pi_n})_\sharp (T_n)]\right)
\]
and, from the arbitrariness of $\delta$, from Corollary \ref{coroll}, together with the fact that ${\rm supp}(T_n)$ is contained in ${\rm span}\{e_1,\ldots,e_{2^n}\}$, we finally get that
\[
\mathds E(({\pi_n})_\sharp S_n)\ge\mathds M^{(2^n,1/2)}\left(Z'[({\pi_n})_\sharp (T_n)]\right)\stackrel{\ref{coroll}}{=}\mathds E\left(({\pi_n})_\sharp T_n\right)=\mathds E(T_n)\,.
\]
\end{proof}
\begin{remark}
In \cite{De_Pauw_Hardt} the authors combine the notion of currents in metric spaces of \cite{Am_Kirch} and currents with coefficients in a group of \cite{White2}. In this setting we can give a meaning to the problem
 of minimizing the mass in a class of currents in $\ell^2$ with coefficients in $\ell^2$ with given boundary. It is possible to associate to the current $T$, defined in \S \ref{subs:constr}, a current $Z$ in $\ell^2$ with 
coefficients in $\ell^2$, in a way which is similar to that explained in Definition \ref{def:pass}. Via a simple calibration argument (the calibration being again the ``identity'' with respect to suitable 
coordinates) one can prove that $Z$ minimizes the mass among all currents in $\ell^2$ with coefficients in $\ell^2$, with boundary $\partial Z$.
\end{remark}
The latter remark suggests the possibility to describe, in general, the continuous version of the Gilbert-Steiner problem as a mass minimization problem among currents in the Euclidean space, with 
coefficients in $\ell^2$. The advantage of this formulation would be, for example, the availability of well known numerical methods for the minimization of a convex functional.

\bibliographystyle{acm}
\bibliography{Bib_A}

\end{document}